\documentclass[12pt,twoside,a4paper]{amsart}
\usepackage[centertags]{amsmath}
\usepackage{amssymb,amscd,amsfonts,latexsym,a4wide,amsthm}

\advance\headheight by 3pt

\renewcommand{\subjclass}[1]{\thanks{\emph{2000 Mathematics Subject Classification:}~#1}}
\renewcommand{\keywords}[1]{\thanks{\emph{Keywords and Phrases:}~#1}}

\newtheorem{thm}{Theorem}[section]

\newtheorem{lemma}[thm]{Lemma}
\newtheorem{prop}[thm]{Proposition}

\theoremstyle{definition}

\theoremstyle{remark}

\numberwithin{equation}{section}

\newcommand{\Cc}{{\mathbb C}}

\newcommand{\Zz}{{\mathbb Z}}
\newcommand{\Qq}{{\mathbb Q}}

\newcommand{\Rr}{{\mathbb R}}

\newcommand{\x}{{\bf x}}

\renewcommand{\leq}{\leqslant}
\renewcommand{\geq}{\geqslant}

\newcommand{\kdots}{,\ldots ,}

\title[On the quantitative Subspace Theorem]
{On the quantitative Subspace Theorem}

\subjclass{11J68, 11J25}
\keywords{Diophantine approximation, Subspace Theorem}

\author[J.-H.~EVERTSE]{Jan-Hendrik~EVERTSE}
\address{J.-H. Evertse,
Universiteit Leiden, Mathematisch Instituut,
Postbus 9512, 2300 RA Leiden,
The Netherlands}
\email{evertse@math.leidenuniv.nl}
\date{\today}

\begin{document}

\maketitle

\begin{abstract}
In this survey we give an overview of recent improvements
upon the Quantitative Subspace Theorem, obtained jointly
with R. Ferretti, which follow from work in \cite{EF-to-appear}.
Further, we give a new gap principle with which we can estimate
the number of subspaces containing the ``small solutions" of the 
systems of inequalities being considered.
As an introduction, we start with a quantitative version
of Roth's Theorem.
\end{abstract}

\section{A quantitative Roth's Theorem}\label{Section1}

Recall that the (absolute) height of an algebraic number $\xi$ of degree $d$
is given by
\[
H(\xi ):=\Big(a\cdot \prod_{i=1}^d\max (1,|\xi^{(i)}|)\Big)^{1/d},
\]
where
$\xi^{(1)}\kdots\xi^{(d)}$ are the conjugates of $\xi$ in $\Cc$
and where $a$ is the positive integer such that the polynomial
$a\cdot\prod_{i=1}^d (X-\xi^{(i)})$ has rational integral coefficients
with gcd $1$.
In particular, if $\xi\in\Qq$, then
$H(\xi )=\max (|x|,|y|)$,
where $x,y$ are coprime integers such that $\xi =x/y$.

Roth's celebrated theorem from 1955 (see \cite{Roth55}) states that
if $\xi$ is any real algebraic number and $\delta$ any real with
$\delta >0$, then the inequality
\begin{equation}\label{1.1}
|\xi -\alpha |\leq H(\alpha )^{-2-\delta}\ \ \mbox{in $\alpha\in\Qq$}
\end{equation}
has only finitely many solutions. Already in 1955, Davenport and Roth
\cite{DR55} computed an upper bound for the number of solutions of
\eqref{1.1}, and their bound was subsequently improved by
Mignotte \cite{Mign74}, Bombieri and van der Poorten \cite{BP88},
and the author \cite{Ev96}. We formulate a slight improvement of the latter
result which follows from the Appendix of \cite{BuEv08}.
We mention that this improvement is obtained by simply going through
the existing methods; its proof did not involve anything new.
We distinguish
between \emph{large} and \emph{small} solutions $\alpha$ of \eqref{1.1},
where a rational number $\alpha$ is called large if
\begin{equation}\label{1.2}
H(\alpha )\geq \max \big( H(\xi ) ,2\big)
\end{equation}
and small otherwise.

\begin{thm}\label{th:1.1}
Let $\xi$ be a real algebraic number of degree $d$ and $0<\delta\leq 1$.
Then the number of large solutions of \eqref{1.1} is at most
\[
2^{25}\delta^{-3}\log (2d)\log \big(\delta^{-1}\log (2d)\big)
\]
and the number of small solutions at most
\[
10\delta^{-1}\log\log\max\big( H(\xi ),4\big).
\]
\end{thm}

The proof of this result can be divided into two parts: a
so-called \emph{interval result} and a \emph{gap principle}.
The interval result may be stated as follows.

\begin{prop}\label{pr:1.2}
Let
\begin{eqnarray*}
&& m:= 1+[25600\delta^{-2}\log (2d)],\ \
\omega := 162m^2\delta^{-1} ,
\\
&& C:= \exp\Big( 3m\binom{d}{2}\delta^{-1}\big(240m^2\delta^{-1}\big)^m
\cdot\log\big( 36H(\xi )\big)\Big).
\end{eqnarray*}
Then there are reals $Q_1\kdots Q_{m-1}$ with
\[
C\leq Q_1<Q_2<\cdots <Q_{m-1}
\]
such that if $\alpha\in\Qq$ is a solution of \eqref{1.1} with $H(\alpha )\geq C$,
then
\[
H(\alpha )\in \bigcup_{i=1}^{m-1} \big[Q_i,Q_i^{\omega}\big).
\]
\end{prop}

The proof is by means of the usual ``Roth machinery."
Assume Theorem \ref{1.2} is false. Then \eqref{1.1} has solutions
$\alpha_1\kdots\alpha_m$ such that $H(\alpha_1)\geq C$ and
$H(\alpha_i)\geq H(\alpha_{i-1})^{\omega}$ for $i=1\kdots m$.
One constructs a polynomial $F(X_1\kdots X_m)$ which has integer
coefficients of small absolute value,
and which is of degree $d_i$ in the variable $X_i$
for $i=1\kdots m$, such that
$H(\alpha_1)^{d_1}\approx\cdots\approx H(\alpha_m)^{d_m}$, and such that
$F$ has large ``index" (some sort of weighted multiplicity) at the point
$(\alpha_1\kdots\alpha_m)$. Then one applies Roth's Lemma
(a non-vanishing result for polynomials) to conclude that $F$ cannot
have large index at $(\alpha_1\kdots\alpha_m)$.
In fact, we use a refinement of Roth's original Lemma from 1955
(see \cite{Ev95})
which was proved by means of the techniques going into the proof of Faltings'
Product Theorem \cite{Falt91}.

The second ingredient is the following very basic gap principle.

\begin{prop}\label{pr:1.3}
Let $Q\geq 2$. Then \eqref{1.1} has at most one solution $\alpha$
such that $Q\leq H(\alpha )<Q^{1+\delta /2}$ and $\alpha >\xi$,
and also at most one solution $\alpha$ such that
$Q\leq H(\alpha )<Q^{1+\delta /2}$ and $\alpha <\xi$.
\end{prop}

\begin{proof} Suppose for instance that \eqref{1.1} has two solutions
$\alpha_1,\alpha_2$ which are both larger than $\xi$,
and $Q\leq H(\alpha_1)\leq H(\alpha_2)<Q^{1+\delta /2}$ for $i=1,2$.
Then
\begin{eqnarray*}
Q^{-2(1+\delta /2)}&<&
\big(H(\alpha_1)H(\alpha_2)\big)^{-1}\leq |\alpha_1-\alpha_2|
\\
&\leq&\max_i |\xi-\alpha_i|\leq H(\alpha_1)^{-2-\delta}\leq Q^{-2-\delta}
\end{eqnarray*}
which is obviously impossible.
\end{proof}

An immediate consequence of this gap principle is that for any $Q\geq 2$,
$E>1$, inequality \eqref{1.1}
has at most $1+\,2\log E/\log (1+\delta /2)$ solutions $\alpha\in\Qq$
with $Q\leq H(\alpha )<Q^E$. Using this fact in combination with
Proposition \ref{pr:1.2}, the deduction of Theorem \ref{1.1}
is straightforward.
\\[0.15cm]

Also in more advanced situations, the general pattern to obtain explicit
upper bounds for the number of solutions 
of certain Diophantine equations or inequalities, 
is first to prove that
the number of solutions is finite by means of an involved
Diophantine approximation method, and second to estimate from above the number
of solutions using a more or less elementary gap principle.
However, there are also many situations 
where we do have at our disposal a method
to prove finiteness for the number of solutions but where we do not have a gap
principle.
So in these situations we know that there are only finitely many solutions, 
but we are not able to estimate their number.

\section{The quantitative Subspace Theorem}\label{Section2}

We generalize the results from Section \ref{Section1} to higher
dimensions.

Let $n\geq 2$ be an integer.
We denote by $\|\cdot \|$ the maximum norm on $\Rr^n$.
Let
\[
L_i=\alpha_{i1}X_1+\cdots +\alpha_{in}X_n\ \ (i=1\kdots n)
\]
be linearly forms with algebraic coefficients $\alpha_{ij}\in\Cc$
which are linearly independent, that is, their coefficient determinant
$\det (L_1\kdots L_n)=\det (\alpha_{ij})$ is non-zero.
Further, let $\delta>0$ and consider the inequality
\begin{equation}\label{2.1}
|L_1({\bf x})\cdots L_n({\bf x})|
\leq
\|{\bf x}\|^{-\delta}\ \ \mbox{in } {\bf x}\in\Zz^n .
\end{equation}
W. Schmidt's celebrated Subspace Theorem from 1972 (see \cite{Schm72})
states that the set of solutions of \eqref{2.1} lies in a union
of finitely many proper linear subspaces of $\Qq^n$.
In 1989,
Schmidt proved \cite{Schm89} a quantitative result,
which in a slightly modified form reads as follows.
\\[0.2cm]
\emph{Suppose that the algebraic numbers $\alpha_{ij}$ have height at most $H$
and degree at most $D$ and that $0<\delta\leq 1$.
Then the solutions of
\[
|L_1({\bf x})\cdots L_n({\bf x})|
\leq
|\det (L_1\kdots L_n)|\cdot
\|{\bf x}\|^{-\delta}\ \ \mbox{in } {\bf x}\in\Zz^n
\]
with
$\|{\bf x}\|\geq \max (2H,n^{2n/\delta})$ lie in a union of at most
$2^{2^{27n\delta^{-2}}}$ proper linear subspaces of $\Qq^n$.}
\\[0.2cm]
This quantitative result has been improved and generalized in various
directions, mainly due to work of Schlickewei and the author.

We now discuss versions of the Subspace Theorem which involve non-archimedean
absolute values and which take their unknowns from algebraic number fields.
All our algebraic number fields considered below are contained in a given
algebraic closure $\overline{\Qq}$ of $\Qq$.

Let $M_{\Qq}:=
\{\infty\}\cup\{ {\rm primes}\}$ denote the set of places of $\Qq$.
We write
$|\cdot |_{\infty}$ for the ordinary absolute value on $\Qq$ and $|\cdot |_p$
($p$ prime number) for the $p$-adic absolute value, normalized such that
$|p|_p=p^{-1}$. Further, we denote by $\Qq_p$ the completion of $\Qq$ at $p$;
in particular, $\Qq_{\infty}=\Rr$.

Let $K$ be an algebraic number field and denote by $M_K$
the set of places of $K$. To every place $v\in M_K$,
we associate an absolute value $|\cdot |_v$ which is such that if $v$ lies
above $p\in M_{\Qq}$, then the restriction of $|\cdot |_v$ to $\Qq$
is $|\cdot |_p^{[K_v:\Qq_p]/[K:\Qq ]}$, where $K_v$ is the completion
of $K$ at $v$.
The absolute value $|\cdot |_v$ can be continued uniquely to the algebraic
closure $\overline{K_v}$ of $K_v$.
The place $v$ is called finite if $v\nmid\infty$,
infinite if $v|\infty$, real if $K_v=\Rr$ and complex if $K_v=\Cc$.
The absolute values thus chosen satisfy the product formula
$\prod_{v\in M_K} |x|_v=1$ for $x\in K^*$.

We define the height (not the standard definition) of
${\bf x}=(x_1\kdots x_n)\in K^n$ by
\[
H({\bf x}):=\prod_{v\in M_K}\max (1,|x_1|_v\kdots |x_n|_v).
\]

Let $S$ be a finite subset of $M_K$, containing all infinite places.
Denote by $O_S=\{ x\in K:\, |x|_v\leq 1$ for $v\in M_K\setminus S\}$
the ring of $S$-integers.
For $v\in S$, let
\[
L_i^{(v)}=\alpha_{i1}^{(v)}X_1+\cdots +\alpha_{in}^{(v)}X_n\ \ (i=1\kdots n)
\]
be linearly independent linear forms with coefficients
$\alpha_{ij}^{(v)}\in\overline{K_v}$
that are algebraic over $K$.

In 1977, Schlickewei \cite{Schl77}
proved that the set of solutions of the inequality
\begin{equation}\label{2.2}
\prod_{v\in S} |L_1^{(v)}({\bf x})\cdots L_n^{(v)}({\bf x})|_v
\leq H({\bf x})^{-\delta}\ \ \mbox{in } {\bf x}\in O_S^n
\end{equation}
is contained in a union of finitely many proper linear subspaces of
$K^n$.

By an elementary combinatorial argument
(see for instance \cite[Section 21]{ES02}),
one can show that every solution ${\bf x}$ of \eqref{2.2}
satisfies one of a finite number of systems of inequalities
\begin{equation}\label{2.3}
|L_i^{(v)}({\bf x})|_v\leq C_v H({\bf x})^{c_{iv}}\ \ (v\in S,\, i=1\kdots n)
\ \ \mbox{in } \x\in O_S^n
\end{equation}
where $C_v>0$ for $v\in S$ and $\sum_{v\in S}\sum_{i=1}^n c_{iv}<0$.
Thus, an equivalent version of Schlickewei's extension of the Subspace Theorem
is the following result which we state for reference purposes:
\\[0.3cm]
{\bf Theorem A.}
\emph{Suppose $C_v>0$ for $v\in S$ and $\sum_{v\in S}\sum_{i=1}^n c_{iv}<0$.
Then the solutions of \eqref{2.3} lie in finitely
many proper linear subspaces of $K^n$.}
\\[0.2cm]

Put
\begin{eqnarray*}
&&s(v):= 1/[K:\Qq ]\ \ \mbox{if $v$ is real,}\
s(v):= 2/[K:\Qq ]\ \ \mbox{if $v$ is complex,}
\\
&&s(v):=0\ \ \mbox{if $v$ is finite.}
\end{eqnarray*}
The following technical conditions on the linear forms $L_i^{(v)}$,
the constants $C_v$ and the exponents $c_{iv}$ will be kept
throughout:
\begin{equation}\label{2.4}
\left\{
\begin{array}{l}
H(\alpha_{ij}^{(v)})\leq H,\ [K (\alpha_{ij}^{(v)}):K ]\leq D\
\mbox{for $v\in S$, $i,j=1\kdots n$;}
\\[0.2cm]
\displaystyle{\#\bigcup_{v\in S}\{ L_1^{(v)}\kdots L_n^{(v)}\}\leq R} ;
\\[0.2cm]
\displaystyle{0<\prod_{v\in S} C_v\leq \prod_{v\in S}|\det (L_1^{(v)}\kdots L_n^{(v)})|_v^{1/n}};
\\[0.2cm]
\displaystyle{\sum_{v\in S}\sum_{i=1}^n c_{iv}\le -\delta\ \mbox{with }
0<\delta\leq 1};
\\[0.2cm]
\max (c_{1v}\kdots c_{nv})=s(v)\ \mbox{for $v\in S$.}
\end{array}\right.
\end{equation}
The following result is an easy consequence of a general
result of Schlickewei and the author \cite[Theorem 2.1]{ES02}:
\\[0.3cm]
{\bf Theorem B.}
\emph{Assume \eqref{2.4}. Then the set of solutions ${\bf x}\in O_S^n$
of \eqref{2.3}
with
\[
H({\bf x})\geq\max (2H,n^{2n/\delta})
\]
is contained in a union of at most
\[
4^{(n+9)^2}\delta^{-n-4}\log (2RD)\log\log (2RD)
\]
proper linear subspaces of $K^n$.}
\\[0.2cm]
In fact, Schlickewei and the author
proved a more general  ``absolute" version where the unknowns
may be algebraic numbers not necessarily belonging to a
fixed number field.

For applications it is important that the upper bound for the number of
subspaces is independent of the field $K$.
The quantity $R$ may be replaced by $ns$, where
$s$ is the cardinality of $S$. But in many cases,
$R$ can be taken independently of $s$. For instance in applications
to linear equations with unknowns from a finitely generated
multiplicative group and to linear recurrence sequences
(see \cite{Schl96}, \cite{ESS02}, \cite{Schm99})
one has to apply the above Theorem with
$L_i^{(v)}\in\{ X_1\kdots X_n, X_1+\cdots +X_n\}$ for $v\in S$,
$i=1\kdots n$, and in that case, one may take $R=n+1$.

Theorem B was the outcome of a development resulting from Schmidt's
quantitative version of the Subspace Theorem mentioned above
and subsequent improvements
and generalizations by Schlickewei and the author
\cite{Schl92}, \cite{Schl96}, \cite{Ev96}, \cite{ES02}.

The proof of Theorem B is basically a quantification of Schmidt's method
of proof of his Subspace Theorem from 1972 (see \cite{Schm72},
\cite{Schm80}).
It consists of geometry of numbers, a construction of an
auxiliary polynomial, and an application of Roth's Lemma.
In 1994, Faltings and W\"{u}stholz \cite{FW94} gave a totally new proof
of the Subspace Theorem. In their proof they did not use geometry of numbers,
and instead of Roth's Lemma they applied the much more powerful
Faltings' Product Theorem.
Another important ingredient of the proof of Faltings and W\"{u}stholz
is a stability theory for multi-filtered vector spaces.
The method of Faltings and W\"{u}stholz also allows
to compute an upper bound for the number of subspaces containing the
solutions of \eqref{2.3}, but this is much larger than the one from
Theorem B. In fact, in the proof of Faltings and W\"{u}stholz one has
to construct global line bundle sections on products of algebraic varieties
of very large degrees (as opposed to Schmidt's proof where one
encounters only linear varieties)
and this leads to poor estimates for the number of subspaces.

However, the upper bound from Theorem B can be improved
further if one combines ideas from Schmidt's method of proof
with ideas from Faltings and W\"{u}stholz.
Essentially, one may follow Schmidt's method of proof,
but replace Schmidt's construction of an auxiliary polynomial
by that of Faltings and W\"{u}stholz, see Section \ref{Section6}
for more details.

In this way, Ferretti and the author \cite{EF-to-appear} obtained the following.
A solution ${\bf x}$ of \eqref{2.3} is called \emph{large} if
\[
H({\bf x})\geq\max (H,n^{2n/\delta})
\]
and \emph{small} otherwise.

\begin{thm}\label{th:2.1}
Assume \eqref{2.4}. Then the set of large solutions of \eqref{2.3}
lies in a union of at most
\[
10^9 2^{2n}n^{14}\delta^{-3}\log (3\delta^{-1}RD)\cdot
\log (\delta^{-1}\log 3RD)
\]
proper linear subspaces of $K^n$.
\end{thm}

So compared with Theorem B, the dependence on $n$ has been brought down
from $c^{n^2}$ to $c^n$, while the dependence on $\delta$ has been
improved from $\delta^{-n-4}$ to $\delta^{-3}(\log \delta^{-1})^2$.
With this improvement, the dependence on $\delta$ is almost as good as
that in the quantitative Roth's Theorem from the previous section.
One might still hope for a further improvement in terms of $n$,
for instance to
something polynomial in $n$, but probably this would require a new
method of proof for the Subspace Theorem.

For the small solutions we have the following elementary result
which is proved in Section \ref{Section5} of the present paper.
Here, in contrast to the large solutions, we do get a dependence
on the field $K$.

\begin{thm}\label{th:2.2}
Assume \eqref{2.4}. Let $d:=[K:\Qq ]$.
Then the set of small solutions of \eqref{2.3} lies in a
union of at most
\[
\delta^{-1}\Big( (10^3n)^{nd}+4n\log\log 4H\Big)
\]
proper linear subspaces of $K^n$.
\\[0.1cm]
In the case $K=\Qq$ this bound can be replaced by
\[
\delta^{-1}\Big( 10^{3n} +4n\log\log 4H\Big).
\]
\end{thm}

It is an open problem whether the bounds in Theorem \ref{th:2.2}
can be replaced by something depending only polynomially on $n$ and/or $d$.
Recent work by Schmidt \cite{Schm09} on Roth's Theorem over number fields
suggests that a polynomial dependence on $d$ should be possible.
  
\section{A refinement of the Subspace Theorem and an interval result}
\label{Section3}

We keep the notation and assumptions from the previous section.
So $K$, $S$, $L_i^{(v)}$ ($v\in S$, $i=1\kdots n$), $\delta$,
have the same meaning as before, and they satisfy \eqref{2.4}.
The following refinement of the Subspace Theorem
follows from work of Faltings and W\"{u}stholz \cite{FW94}
and Vojta \cite{Voj89} but there is a heavy overlap with ideas
of Schmidt \cite{Schm93}.
\\[0.3cm]
{\bf Theorem C.}
\emph{
There is a proper linear subspace $U_0$ of $K^n$, such that
\eqref{2.3} has only finitely many solutions outside $U_0$.
\\[0.1cm]
This space $U_0$ can be determined effectively. Moreover, it can be chosen
from a finite collection, which depends only on the linear forms
$L_i^{(v)}$ ($v\in S$, $i=1\kdots n$) and is independent of the constants
$C_v$ and the exponents $c_{iv}$.}
\\[0.2cm]
The first part giving the mere existence of $U_0$ is Theorem 9.1 of \cite{FW94}.
The second part follows from \cite{Voj89}.
\\[0.15cm]

We first give a description of the space $U_0$ occurring
in Theorem 9.1 of \cite{FW94}, where we have translated 
Faltings' and W\"{u}stholz' terminology into ours.
Let $v\in M_K$.
Two linear forms $L=\sum_{i=1}^n \alpha_i X_i$ and $M=\sum_{i=1}^n \beta_i X_i$
with coefficients in $\overline{K_v}$ are said to be conjugate over $K_v$
if there is an automorphism $\sigma$
of $\overline{K_v}$ over $K_v$ such that $\sigma (\alpha_i )=\beta_i$
for $i=1\kdots n$.
Given $v\in M_K$ and a system of linear forms $L_1\kdots L_r$
with coefficients in $\overline{K_v}$, this system is called
$v$-symmetric if with any linear form in the system, also all its conjugates
over $K_v$ belong to this system.

Given a linear subspace $U$ of $K^n$ and
linear forms $L_1\kdots L_r$ with coefficients generating a field extension $F$
of $K$, we say that $L_1\kdots L_r$ are linearly independent on $U$
if there is no non-trivial linear combination of $L_1\kdots L_r$
with coefficients in $F$ that vanishes identically on $U$.

For each $v\in S$, we obtain a $v$-symmetric system 
$L_1^{(v)}\kdots L_{n_v}^{(v)}$,
consisting of the linear
forms $L_1^{(v)}\kdots L_n^{(v)}$ from \eqref{2.3} and their conjugates
over $K_v$.
Using $|L({\bf x})|_v=|M({\bf x})|_v$
for any ${\bf x}\in K^n$ and any linear forms $L,M$ with coefficients
in $\overline{K_v}$ which are
conjugate over $K_v$, we see that \eqref{2.3}
is equivalent to the system of inequalities
\begin{equation}\label{3.1}
|L_i^{(v)}({\bf x})|_v\leq C_v H({\bf x})^{c_{iv}}\ (v\in S,\, i=1\kdots n_v)
\ \ \mbox{in } {\bf x}\in O_S^n .
\end{equation}

Now for any linear subspace $U$ of $K^n$ and any $v\in S$,
define $\nu_v(U)=0$ if $U=(0)$ and
\[
\nu_v(U):=\min c_{i_1,v}+\cdots +c_{i_u,v}
\]
if $U\not= (0)$, where $u=\dim U$,
and the minimum is taken over all subsets $\{ i_1\kdots i_u\}$
of $\{ 1\kdots n_v\}$ of cardinality $u$
such that $L_{i_1}^{(v)}\kdots L_{i_u}^{(v)}$
are linearly independent on $U$. Further, define
\[
\nu (U) :=\sum_{v\in S} \nu_v(U),
\]
and, if $U\not=K^n$,
\[
\mu (U) := \frac{\nu (K^n)-\nu (U)}{n-\dim U}.
\]
Let $\mu_0$ be the mimimum of the quantities $\nu (U)$, taken over all
proper linear subspaces $U$ of $K^n$.

Now one can show that there is a unique proper linear subspace $U_0$ of $K^n$,
which is the one from Theorem C, such that
\begin{equation}\label{3.2}
\left\{
\begin{array}{l}
\mu (U_0)=\mu_0;
\\
U_0\subseteq U\ \mbox{for every linear subspace
$U$ of $K^n$ with $\mu (U)=\mu_0$.}
\end{array}
\right.
\end{equation}

It is important to remark, that
Theorem C can be deduced from the apparently weaker
Theorem A. The argument is roughly as follows.
First assume that $U_0=({\bf 0})$.
(In this case, following the terminology of Faltings and W\"{u}stholz,
system \eqref{2.3} is called \emph{semistable}.)
This assumption implies that if $U$ is any linear subspace of $K^n$ 
of dimension at least $2$, 
then Theorem A is applicable to the restriction of \eqref{2.3} to $U$,
and thus, the solutions of \eqref{2.3} in $U$
lie in a finite union of proper linear subspaces of $U$.
Now by induction, it follows easily that \eqref{2.3} has only finitely
many solutions.

If $U_0\not=({\bf 0})$, one may derive from \eqref{2.3} a semistable
system of inequalities, with solutions from the quotient vector
space $K^n/U_0$. We infer that the solutions of \eqref{2.3} outside $U_0$
lie in finitely many cosets modulo $U_0$.
Then one completes the proof by showing 
that each coset contains only finitely many solutions.

The space $U_0$ can be determined effectively in principle
using a combinatorial algorithm based on ideas of Vojta \cite{Voj89}.
In fact,
let $M_1\kdots M_t$ be the conjugates in $\overline{\Qq}[X_1\kdots X_n]$ 
of the linear
forms $L_i^{(v)}$ ($v\in S$, $i=1\kdots n$).
Let $F$ be the extension of $K$ generated by the coefficients of
$M_1\kdots M_t$. 
Define the $F$-vector spaces
$H_i:=\{ {\bf x}\in L^n:\, M_i({\bf x})=0\}$ ($i=1\kdots t$).
From ideas of Vojta \cite{Voj89} it follows that $U_0\otimes_K F$ 
can be obtained by an algorithm taking as input
the spaces $H_1\kdots H_t$ and applying repeatedly the operations
$+$ (sum of two vector spaces) and $\cap$ (intersection) to two
previously obtained spaces. The number of steps of this algorithm
is bounded above effectively in terms of $t$ only.
 
Alternatively, from an auxiliary result in \cite{EF-to-appear} 
it follows that $U_0$ has a basis, consisting of vectors 
of which the coordinates 
have heights at most $\big( \sqrt{n}H\big)^{4^n}$,
where $H$ is given by \eqref{2.4}.

The special case that $L_i^{(v)}\in\{ X_1\kdots X_n,\, X_1+\cdots +X_n\}$ 
for $v\in S$, $i=1\kdots n$ is of particular importance for applications.
It is shown in \cite{EF-to-appear} that in this case we have
\[
U_0=\left\{ {\bf x}=(x_1\kdots x_n)\in K^n:\, \sum_{j\in I_i} x_j=0\ 
\mbox{for } i=1\kdots t\right\}
\]
where $I_1\kdots I_t$ are certain pairwise disjoint subsets of
$\{ 1\kdots n\}$.
\\[0.15cm]

The solutions of \eqref{2.3} outside $U_0$ can not be determined effectively.
Moreover, it is also beyond reach to estimate the number of solutions
outside $U_0$. But Ferretti and the author \cite{EF-to-appear} 
proved the following more precise
version of Theorem C which may be considered as an analogue
of the interval result Proposition \ref{pr:1.2}.

\begin{thm}\label{th:3.1}
Assume \eqref{2.4}. Put
\[
m:=\big[ 10^8 2^{2n}n^{14}\delta^{-2}\log (3\delta^{-1}RD)\big],\ \ \
\omega := 3n\delta^{-1}\log 3RD .
\]
Then there are reals $Q_1\kdots Q_m$ with
\[
\max (2H, n^{2n/\delta})\leq Q_1<Q_2<\cdots <Q_m
\]
such that for every solution ${\bf x}\in O_S^n$ of \eqref{2.3} outside $U_0$
we have
\[
H({\bf x})<\max (2H,n^{2n/\delta})\ \mbox{or }
H({\bf x})\in\bigcup_{i=1}^m\big[ Q_i,Q_i^{\omega}\big).
\]
\end{thm}

In \cite{EF-to-appear} we proved a more general absolute result where the
unknowns are taken from $\overline{\Qq}$ instead of $K$.

\section{Gap principles}\label{Section5}

In this section we state and prove two gap principles.
Further, we deduce Theorems \ref{th:2.1} and \ref{th:2.2}.
Theorem \ref{th:2.1} is a consequence of Theorem \ref{th:3.1} and our
first gap principle, while Theorem \ref{th:2.2} follows from our
second gap principle.

We keep the notation introduced before. Further, we put
\[
\Delta_v:=|\det (L_1^{(v)}\kdots L_n^{(v)})|_v\ \ \mbox{for $v\in S$.}
\]

We state our first gap principle. This result is well-known but
we have included a proof for convenience of the reader.

\begin{prop}\label{pr:3.2}
Assume \eqref{2.4}. Let $Q\geq  n^{2n/\delta}$. Then the set of solutions
${\bf x}\in O_S^n$ of \eqref{2.3} with
\[
Q\leq \|{\bf x}\|<Q^{1+\delta /2n}
\]
is contained in a single proper linear subspace of $K^n$.
\end{prop}

\begin{proof}
Let $\mathcal{T}$ denote the set of solutions ${\bf x}\in O_S^n$
to \eqref{2.3} with
\[
Q\leq H({\bf x})<Q^{1+\delta /2n}\ \mbox{for } i=1\kdots n.
\]
Notice that for ${\bf x}\in\mathcal{T}$ we have, by the last condition
of \eqref{2.4},
\begin{equation}\label{3.12}
|L_i^{(v)}({\bf x})|_v\leq C_vH({\bf x})^{s(v)+(c_{iv}-s(v))}
\leq C_vQ^{c_{iv}+s(v)\delta /2n}\ \mbox{for $i=1\kdots n$.}
\end{equation}
Take ${\bf x}_1\kdots {\bf x}_n\in\mathcal{T}$.
First let $v$ be an infinite place of $K$. Then $|\cdot |_v$ can be extended to
$\overline{K_v}=\Cc$ and for this extension we have
$|\cdot |_v=|\cdot |^{s(v)}$. Now by Hadamard's inequality,
\begin{eqnarray}
\label{3.13}
|\det ({\bf x}_1\kdots {\bf x}_n)|_v &=&\Delta_v^{-1}\cdot
|\det \big(L_i^{(v)}({\bf x}_j)\big)_{i,j}|_v
\\
\nonumber
&\leq& (n^{n/2})^{s(v)}\Delta_v^{-1}
\prod_{i=1}^n\max_{j=1}^n |L_i^{(v)}({\bf x}_j)|_v
\\
\nonumber
&\leq& (n^{n/2})^{s(v)}\Delta_v^{-1}C_v^n
Q^{(\sum_{i=1}^n c_{iv})+s(v)\delta /2}.
\end{eqnarray}
For finite $v\in S$ we have by a similar argument, but now using
$s(v)=0$ and the ultrametric inequality instead of Hadamard's inequality,
\begin{equation}\label{3.20}
|\det ({\bf x}_1\kdots {\bf x}_n)|_v\leq
\Delta_v^{-1}C_v^n
Q^{\sum_{i=1}^n c_{iv}},
\end{equation}
while for the places $v$ outside $S$ we have, trivially,
\begin{equation}\label{3.21}
|\det ({\bf x}_1\kdots {\bf x}_n)|_v \leq 1.
\end{equation}
Now taking the product over $v\in M_K$ and using \eqref{2.4},
$\sum_{v|\infty} s(v)=1$, \eqref{3.13}--\eqref{3.21}
and our assumption $Q>n^{2n/\delta}$
we obtain
\begin{eqnarray*}
\prod_{v\in M_K} |\det ({\bf x}_1\kdots {\bf x}_n)|_v&\leq&
n^{n/2}\prod_{v\in S} (\Delta_v^{-1}C_v^n)\cdot
Q^{(\delta /2)+\sum_{v\in S}\sum_{i=1}^n c_{iv}}
\\
&\leq&  n^{n/2}Q^{-\delta /2}<1,
\end{eqnarray*}
and so, $\det ({\bf x}_1\kdots {\bf x}_n)=0$ by the product formula.
Hence ${\bf x}_1\kdots {\bf x}_n$ are linearly dependent.
This holds for arbitrary ${\bf x}_1\kdots {\bf x}_n\in\mathcal{T}$.
Therefore, $\mathcal{T}$ is contained in a single proper linear subspace of $K^n$.
\end{proof}

\begin{proof}[Proof of Theorem \ref{th:2.1}]
According to Theorem \ref{th:3.1}, for the large solutions ${\bf x}$
of \eqref{2.3} outside $U_0$ we have
$H({\bf x})\in \mathcal{U}:=\bigcup_{i=1}^m [Q_i,Q_i^{\omega})$.
We have to cover $\mathcal{U}$ by intervals of the shape
$[Q,Q^{1+\delta /2n})$ and then apply Proposition \ref{pr:3.2}.
It is not difficult to show that $\mathcal{U}$
is contained in a union of not more than
\[
m\left( 1+\left[ \frac{\log \omega}{\log (1+\delta /2n )}\right]\right)
\]
intervals of the shape $[Q,Q^{1+\delta /2n})$.
By Proposition \ref{pr:3.2}, this quantity, with one added to it
to take care of the space $U_0$, is then an upper bound for the number of
subspaces containing the large solutions of \eqref{2.3}.
This is bounded above by the quantity in Theorem \ref{th:2.1}.
\end{proof}

We now deduce a gap principle to deal with the small solutions of
\eqref{2.3} which is more intricate than the one deduced above.

\begin{prop}\label{pr:3.3}
Let $d:= [K:\Qq ]$ and $Q\geq 1$.
Then the set of solutions
${\bf x}\in O_S^n$ of \eqref{2.3} with
\[
Q\leq H({\bf x})< 2Q^{1+\delta /2n}
\]
is contained in a union of at most
\[
(90n)^{nd}
\]
proper linear subspaces of $K^n$.
\\[0.1cm]
If $K=\Qq$ this upper bound can be replaced by
\[
200^n.
\]
\end{prop}

In the proof we need a number of lemmas.
For ${\bf y}=(y_1\kdots y_n)\in\Cc^n$, define
$\|{\bf y}\|:=\max (|y_1|\kdots |y_n|)$.

\begin{lemma}\label{le:3.5}
Let $M\geq 1$.
We can partition $\Cc^n$ into at most $(20n)^nM^2$ subsets,
such that
for any ${\bf y}_1\kdots {\bf y}_n\in\Cc^n$ belonging to the same subset,
\begin{equation}\label{3.11}
|\det ({\bf y}_1\kdots {\bf y}_n)|\leq M^{-1}\|{\bf y}_1\|\cdots \|{\bf y}_n\|.
\end{equation}
\end{lemma}

\begin{proof}
We can express any non-zero ${\bf y}\in\Cc^n$
uniquely as $\lambda \cdot {\bf z}$, where $\lambda$ is a complex number
with $|\lambda |=\|{\bf y}\|$, and where ${\bf z}=(z_1\kdots z_n)\in\Cc^n$ 
with $\|{\bf z}\|=1$ and
with $|z_j|<1$ for $j<i$ and $z_i=1$ for some $i\in\{ 1\kdots n\}$.
For $j=1\kdots n$, $j\not= i$ we write $z_j=u_j+\sqrt{-1}v_j$ with
$u_j,v_j\in\Rr$.
Further, we express ${\bf 0}$ as $0\cdot {\bf z}$ with
${\bf z}=(1,0\kdots 0)$, and put $u_j=0,v_j=0$ for $j=2\kdots n$.
Thus, with every ${\bf y}\in\Cc^n$ we associate a unique index
$i\in\{ 1\kdots n\}$ and a unique vector
${\bf w}=(u_j,v_j:\, j\not= i)\in [-1,1]^{2n-2}$.

Let $K:= \big( M\cdot n^{n/2}\big)^{1/(n-1)}$.
We divide the $(2n-2)$-dimensional cube
$[-1,1]^{2n-2}$ into at most $([2\sqrt{2}K]+1)^{2n-2}$ subcubes of size
at most $(\sqrt{2}\cdot K)^{-1}$. Then we divide $\Cc^n$ into at most
$n ([2\sqrt{2}\cdot K]+1)^{2n-2}$ classes such that two vectors ${\bf y}$
belong to the same class if the indices $i$ associated with them are equal,
and the vectors ${\bf w}$ associated with them
belong to the same subcube. Notice that the number of classes is bounded
above by
\[
n\Big(2\sqrt{2}\cdot\big( M\cdot n^{n/2}\big)^{1/(n-1)}+1\Big)^{2n-2}
\leq (20n)^n M^2.
\]

Now let ${\bf y}_1\kdots {\bf y}_n$ belong to the same class.
For $k=1\kdots n$, write ${\bf y}_k=\lambda_k {\bf z}_k$ as above
and let ${\bf w}_k$ be the corresponding vector from $[-1,1]^{2n-2}$.
Since ${\bf w}_1\kdots {\bf w}_n$ belong to the same subcube we have
\[
\|{\bf z}_k-{\bf z}_1\|\leq \sqrt{2}\cdot \|{\bf w}_k-{\bf w}_1\|\leq K^{-1}
\]
for $k=2\kdots n$. Hence, using Hadamard's inequality,
\begin{eqnarray*}
|\det ({\bf z}_1\kdots {\bf z}_n)|
&=&
|\det ({\bf z}_1, {\bf z}_2-{\bf z}_1\kdots {\bf z}_n-{\bf z}_1)|
\\
&\leq& n^{n/2}(K^{-1})^{n-1}=M^{-1}
\end{eqnarray*}
which implies
\begin{eqnarray*}
|\det ({\bf y}_1\kdots {\bf y}_n)|&=&|\lambda_1\cdots \lambda_n|\cdot
|\det ({\bf z}_1\kdots {\bf z}_n)|
\\
&\leq& M^{-1}\|{\bf y}_1\|\cdots \|{\bf y}_n\|.
\end{eqnarray*}
This completes our proof.
\end{proof}

\begin{lemma}\label{le:3.3}
Let $D$ be a positive real, and let $\mathcal{S}$ be a subset of $\Zz^n$
such that
\[
|\det ({\bf x}_1\kdots {\bf x}_n)|\leq D\ \
\mbox{for } {\bf x}_1\kdots {\bf x}_n\in\mathcal{S}.
\]
Then $\mathcal{S}$ is contained in a union of at most
\[
100^nD^{1/(n-1)}
\]
proper linear subspaces of $\Qq^n$.
\end{lemma}

\begin{proof} This is Lemma 5 of \cite{Ev00}.
\end{proof}

We deduce the following consequence.

\begin{lemma}\label{le:3.4}
Let $D_v$ ($v\in M_{\Qq}$) be positive reals such that $D_v=1$ for all but
finitely many $v$ and put $D:=\prod_{v\in M_{\Qq}} D_v$.
Let $\mathcal{T}$ be a subset of $\Qq^n$ such that
\begin{equation}\label{3.4}
|\det ({\bf x}_1\kdots {\bf x}_n)|_v\leq D_v\ \mbox{for }
v\in M_{\Qq},\,\, {\bf x}_1\kdots {\bf x}_n\in\mathcal{T}.
\end{equation}
Then $\mathcal{T}$ is contained in a union of at most
\begin{equation}\label{3.x}
100^nD^{1/(n-1)}
\end{equation}
proper linear subspaces of $\Qq^n$.
\end{lemma}

\begin{proof}
Without loss of generality we assume that $\mathcal{T}$ is not contained
in a proper linear subspace of $\Qq^n$.
Further, without loss of generality we assume that for every finite
place $v$ of $\Qq$,
\[
D_v=\max\{ |\det ({\bf x}_1\kdots {\bf x}_n)|_v:\,
{\bf x}_1\kdots {\bf x}_n\in\mathcal{T}\}.
\]
Indeed, if the maximum were $D_v'<D_v$, we could replace $D_v$ by $D_v'$
without strengthening \eqref{3.4}, and replace \eqref{3.x} by a smaller
upper bound.

Fix a finite place $v$ and let $\Zz_v:=\{ x\in\Qq :\, |x|_v\leq 1\}$,
i.e., $\Zz_v$ is the localization of $\Zz$ at $v$.
Choose
${\bf y}_1\kdots {\bf y}_n\in\mathcal{T}$ such that
$|\det({\bf y}_1\kdots {\bf y}_n)|_v=D_v$, and let
$\mathcal{M}_v$ denote the $\Zz_v$-module generated by
${\bf y}_1\kdots {\bf y}_n$.
Now if ${\bf x}\in\mathcal{T}$, then ${\bf x}=\sum_{i=1}^n u_i{\bf y}_i$
with $u_1\kdots u_n\in\Qq$. We can express $u_i$ as a quotient of two
determinants, where in the denominator we have
$\det({\bf y}_1\kdots {\bf y}_n)$, and in the numerator the determinant
obtained by replacing ${\bf y}_i$ by ${\bf x}$. Using \eqref{3.4},
this implies that $|u_i|_v\leq 1$ for $i=1\kdots n$.
Hence $\mathcal{T}$ is contained in $\mathcal{M}_v$.

Applying this for every finite place $v$,
we infer that $\mathcal{T}$ is contained
in $\mathcal{M}:=\bigcap_{v\not=\infty} \mathcal{M}_v$,
where the intersection is
over all finite places. The set $\mathcal{M}$
is a lattice of rank $n$ in $\Qq^n$
of determinant $\Delta :=\big(\prod_{v\not=\infty} D_v\big)^{-1}$.
Choose a basis ${\bf z}_1\kdots {\bf z}_n$ of $\mathcal{M}$.
Then $|\det ({\bf z}_1\kdots {\bf z}_n)|=\Delta$.
Define the linear map
$\varphi:\,{\bf u}=(u_1\kdots u_n)\mapsto\sum_{i=1}^n u_i{\bf z}_i$
and let $\mathcal{S}:=\varphi^{-1}(\mathcal{T})$.
Then $\mathcal{S}\subseteq\Zz^n$ and for any
${\bf u}_1\kdots {\bf u}_n\in\mathcal{S}$ we have
\begin{eqnarray*}
|\det ({\bf u}_1\kdots {\bf u}_n)|&=&\Delta^{-1}\cdot
|\det (\varphi({\bf u}_1)\kdots \varphi ({\bf u}_n))|
\\
&\leq&\Delta^{-1}D_{\infty} =\prod_{v\in M_{\Qq}} D_v=D.
\end{eqnarray*}
Now by Lemma \ref{le:3.3}, the set $\mathcal{S}$, and hence also
$\mathcal{T}$, is contained in a union of not more than $100^nD^{1/(n-1)}$
proper linear subspaces of $\Qq^n$.
\end{proof}

We leave as an open problem to generalize the above Lemma
to arbitrary algebraic number fields.

\begin{proof}[Proof of Proposition \ref{pr:3.3}]
We start with the case that $K$ is an arbitrary number field.
Let $\mathcal{T}'$ be the set of solutions ${\bf x}\in O_S^n$ of \eqref{2.3}
with $Q\leq H({\bf x})<2Q^{1+\delta /2n}$. Completely analogously to
\eqref{3.12} we have for ${\bf x}\in\mathcal{T}'$, $v\in S$, $i=1\kdots n$,
\begin{equation}\label{3.14}
|L_i^{(v)}({\bf x})|_v\leq 2^{s(v)}C_vQ^{c_{iv}+s(v)\delta /2n}.
\end{equation}

For ${\bf x}\in\mathcal{T}'$ and any infinite place $v$ of $K$,
define the vector
\[
\varphi_v({\bf x}):=\big( Q^{-c_{1v}/s(v)}L_1^{(v)}({\bf x})\kdots
Q^{-c_{nv}/s(v)}L_n^{(v)}({\bf x})\big).
\]
Notice that for each infinite place $v$ of $K$ we have
$\varphi_v({\bf x})\in\Cc^n$.
Put $M:= (9/2)^{n/2}$. By Lemma \ref{le:3.5},
and since $K$ has at most $d$ infinite places, we can partition
$\mathcal{T}'$ into at most
\[
(20n)^{nd}M^{2d}\leq (90n)^{nd}
\]
classes, such that if ${\bf x}_1\kdots {\bf x}_n$ belong to the same class,
then for each infinite place $v$,
\[
|\det \big(\varphi_v({\bf x}_1)\kdots\varphi_v({\bf x}_n)\big)|\leq
M^{-1}\prod_{i=1}^n\|\varphi_v({\bf x})\| .
\]

We show that the set of elements of $\mathcal{T}'$ from a given class
is contained in a proper linear subspace of $K^n$, that is,
that any $n$ elements of $\mathcal{T}'$
from the same class have determinant $0$. So let ${\bf x}_1\kdots {\bf x}_n$
be elements of $\mathcal{T}'$ from the same class.
Then by \eqref{3.14} and what we just proved, we have for every infinite place
$v$ of $K$, using $|\cdot |_v=|\cdot |_v^{s(v)}$ on $\overline{K_v}=\Cc$,
\begin{eqnarray*}
|\det ({\bf x}_1\kdots {\bf x}_n)|_v&=&\Delta_v^{-1}Q^{\sum_{i=1}^n c_{iv}}
\cdot
|\det \big(\varphi_v({\bf x}_1)\kdots\varphi_v({\bf x}_n)\big)|^{s(v)}
\\
&\leq& \Delta_v^{-1}Q^{\sum_{i=1}^n c_{iv}}M^{-s(v)} \prod_{i=1}^n\|\varphi_v({\bf x})\|^{s(v)}
\\
&\leq&
\Delta_v^{-1}Q^{\sum_{i=1}^n c_{iv}}M^{-s(v)} C_v^n2^{ns(v)}Q^{s(v)\delta /2},
\end{eqnarray*}
which, thanks to our choice of $M$, yields
\[
|\det ({\bf x}_1\kdots {\bf x}_n)|_v<
\Delta_v^{-1}C_vQ^{(\sum_{i=1}^n c_{iv})+s(v)\delta /2}.
\]
For the finite places $v\in S$ we have \eqref{3.20} and for the places
$v$ outside $S$, \eqref{3.21}. By taking the product over $v\in M_K$,
using \eqref{2.4}, we obtain
\[
\prod_{v\in M_K} |\det ({\bf x}_1\kdots {\bf x}_n)|_v<
\prod_{v\in S} (\Delta_v^{-1}C_v^n)
Q^{\sum_{v\in S}\sum_{i=1}^n c_{iv}+(\delta /2)}\leq Q^{-\delta /2}\leq 1.
\]
Now the product formula implies indeed that for any
${\bf x}_1\kdots {\bf x}_n$ in the same class we have
$\det ({\bf x}_1\kdots {\bf x}_n)=0$. This proves Proposition \ref{pr:3.3}
in the case that $K$ is an arbitrary algebraic number field.

Now let $K=\Qq$. Let ${\bf x}_1\kdots {\bf x}_n\in\mathcal{T}'$.
First let $v=\infty$ be the infinite place of $\Qq$.
Notice that $s(\infty )=1$.
Then using \eqref{3.14} we obtain in a similar manner as \eqref{3.13},
\[
|\det ({\bf x}_1\kdots {\bf x}_n)|_{\infty}\leq
n^{n/2}\Delta_{\infty}^{-1} C_{\infty}^n\cdot
2Q^{(\sum_{i=1}^n c_{i\infty})+\delta /2}.
\]
For the finite places in $S$ and for the places outside $S$ we have
\eqref{3.20}, \eqref{3.21}. Now using Lemma \ref{le:3.4},
\eqref{2.4},
we infer that $\mathcal{T}'$ is contained in a union of at most
\begin{eqnarray*}
&&100^n\Big( 2n^{n/2}
\prod_{v\in S} (\Delta_v^{-1}C_v^n)\,\cdot\,
Q^{\sum_{v\in S}\sum_{i=1}^n c_{iv}+(\delta /2)}\Big)^{1/(n-1)}
\\
&&\qquad\qquad\qquad\leq
100^n\big( 2n^{n/2}\big)^{1/(n-1)} < 200^n
\end{eqnarray*}
proper linear subspaces of $\Qq^n$. This completes our proof.
\end{proof}

\begin{proof}[Proof of Theorem \ref{2.2}]
Let $K$ be an arbitrary algebraic number field of degree $d$.
We divide the solutions into consideration into those with 
$H({\bf x})\in I_1$ and those with $H({\bf x})\in I_2$, where
\[
I_1=\big[n^{2n/\delta},\max (2H,n^{2n/\delta})\big),\ \ \ 
I_2=\big[ 1, n^{2n/\delta}\big).
\]
We have $I_1\subseteq \bigcup_{h=0}^{A-1} \big[ Q_h,Q_h^{1+\delta/2n}\big)$,
where
\begin{eqnarray*}
&&Q_h=(n^{2n/\delta})^{(1+\delta/2n)^h}\ \ (h=0,1,2,\ldots),
\\
&&A=1+\left[\frac{\log\Big(\log\max (2H,n^{2n/\delta})/\log n^{2n/\delta}\Big)}
{\log (1+\delta/2n)}\right]\leq 4n\delta^{-1}\log\log 4H.
\end{eqnarray*}
So by Proposition \ref{pr:3.2}, the solutions ${\bf x}\in O_S^n$ of \eqref{2.3}
with $H({\bf x})\in I_1$
lie in a union of at most $A$ proper linear subspaces of $K^n$.

Next, we have $I_2\subseteq \bigcup_{h=0}^{B-1} 
\big[ Q_h, 2Q_h^{1+\delta /2n}\big)$, where
\begin{eqnarray*}
&&Q_h =2^{\gamma_h}\ \mbox{with } 
\gamma_h=\frac{2n}{\delta}\Big( (1+(\delta /2n))^h\,-1\Big)\ \ 
(h=0,1,2,\ldots ),
\\
&&
B= 1+\left[\frac{\log(1+\log n/\log 2)}{\log(1+\delta /2n)}\right]
\leq 4n\delta^{-1}\log (3\log n).
\end{eqnarray*}
So by Proposition \ref{pr:3.3}, the solutions ${\bf x}\in O_S^n$
of \eqref{2.3} with $H({\bf x})\in I_2$ lie in a union of at most
$(90n)^{nd}B$ proper linear subspaces of $K^n$. 

We conclude that the number of subspaces containing the solutions 
${\bf x}\in O_S^n$ of \eqref{2.3}
with $H({\bf x})\leq\max (2H,n^{2n/\delta})$ is bounded above by
\[
A+(90n)^{nd}B\leq \delta^{-1}\big( (10^3n)^{nd}+4n\log\log 4H\big).
\]
In the case $K=\Qq$ we have a similar
computation, replacing $(90n)^{nd}$ by $200^n$.
\end{proof}

\section{On the number of solutions outside the exceptional subspace $U_0$}
\label{Section4}

It seems to be a very difficult open problem to give an upper bound
for the number of solutions of \eqref{2.3} lying outside the exceptional
subspace $U_0$ from Theorem C. To obtain such a bound we would have
to combine the interval result Theorem \ref{th:3.1} with some
strengthening
of the gap principle Proposition \ref{pr:3.2} giving an upper bound for
the number of solutions ${\bf x}$ with $Q\leq H({\bf x})< Q^{1+\delta /2n}$
instead of the number of subspaces containing these solutions.
But this seems to be totally out of reach.
However, such a strong gap principle may exist in certain applications
where one considers solutions ${\bf x}$ with additional constraints,
and then it may be possible to estimate
from above the number of such restricted solutions.

In 1990, Schmidt \cite{Schm89b} gave an example of a system 
of inequalities \eqref{2.3} which is known to have finitely many solutions,
but which is such that 
from any explicit upper bound for the number of solutions
of this system one can derive a very strong
\emph{effective} finiteness result for some related system of Diophantine inequalities.

We give another such example, which is a modification of a result from
Hirata-Kohno and the author \cite{EvHK02}.
We consider the inequality
\begin{eqnarray}\label{4.1}
&&|x_1+x_2\xi +x_3\xi^2|\leq H({\bf x})^{-2-\delta}
\\
\nonumber
&&\qquad\qquad\mbox{in }
{\bf x}=(x_1,x_2,x_3)\in\Zz^3\ \mbox{with gcd}(x_1,x_2,x_3)=1,
\end{eqnarray}
where $\xi$ is a real algebraic number of degree $\geq 3$
and where $\delta >0$. By augmenting this
single inequality with the two trivial inequalities
\[
|x_2|\leq H({\bf x}),\ \ \ |x_3|\leq H({\bf x})
\]
we obtain a system of type \eqref{2.3}.
Since $\xi$ has degree at least $3$, the linear form
$X_1+X_2\xi +X_3\xi^2$ does not vanish identically on any non-zero non-linear
subspace $U$ of $\Qq^3$. Consequently, if $U$ is a
linear subspace of $\Qq^3$ of dimension $k>0$
we have $\nu (U)=-2-\delta +k-1$.
Hence
\[
\mu (U)=\frac{\nu (\Qq^3 )-\nu (U)}{3-\dim U}=1
\]
if $U\not= (0)$ and $\mu ((0))=-\delta /3<1$.
So according to the description of $U_0$ in Section \ref{Section3},
we have $U_0=({\bf 0})$ and by Theorem C, \eqref{4.1} has only finitely
many solutions.
(This can also be deduced directly from Theorem A).

We prove the following Proposition.

\begin{prop}\label{pr:4.1}
Let $N$ be an upper bound for the number of solutions of \eqref{4.1}.
Then for every $\alpha\in\Qq$ we have
\begin{equation}\label{4.2}
|\xi -\alpha |\geq 2^{-2-\delta}(1+|\xi |)^{-1}N^{-3-\delta}\cdot
H(\alpha )^{-3-\delta}.
\end{equation}
\end{prop}

One of the most wanted achievements in Diophantine approximation would be
to prove an effective version of Roth's Theorem,
i.e., an inequality of the shape
\[
|\xi -\alpha |\geq C(\xi ,\delta )H(\alpha )^{-2-\delta}\
\mbox{for $\alpha\in\Qq$}
\]
with some effectively computable constant $C(\xi ,\delta )>0$.
Our Proposition implies that from an explicit upper bound for the
number of solutions of \eqref{4.1} one would be able to deduce
an effective inequality with instead of an exponent $2+\delta$ an
exponent $3+\delta$.
Save some special cases, such a result is much stronger 
than any of the effective results
on the approximation of algebraic numbers by rationals
that have been obtained so far.

\begin{proof} Let $\alpha$ be a rational number.
We can express $\alpha$ as $\alpha =r/s$, where
$r,s$ are rational integers with $s>0$, gcd$(r,s)=1$.
Thus, $H(\alpha )=\max (|r|,|s|)$.
Let $u$ be an integer with
\begin{equation}\label{4.3}
|u|
\leq
\Big(2^{2+\delta}(1+|\xi |)\cdot
|\xi -\alpha |\cdot H(\alpha )^{3+\delta}\Big)^{-1/(3+\delta )}.
\end{equation}
We assume that the right-hand side is at least $1$; otherwise \eqref{4.2}
follows at once.

Define the vector ${\bf x}=(x_1,x_2,x_3)$ by $x_1+x_2X+x_3X^2=(u+X)(r-sX)$.
Then ${\bf x}\in\Zz^3$, gcd$(x_1,x_2,x_3)=1$ and by \eqref{4.3},
\begin{eqnarray*}
|x_1+x_2\xi+x_3\xi^2| &=& |u+\xi |\cdot |r-s\xi |
\\
&\leq& (1+|\xi |)\max (1,|u|)\max (|r|,|s|)\cdot |\xi -\alpha |
\\
&\leq& \Big(2\max (1,|u|)\max (|r|,|s|)\Big)^{-2-\delta}\leq
H({\bf x})^{-2-\delta}.
\end{eqnarray*}
Thus, each integer $u$ with \eqref{4.3}
gives rise to a solution of \eqref{4.1}.
Consequently, the number of solutions of \eqref{4.1}, and hence $N$,
is bounded from below by the right-hand side of \eqref{4.3}.
Now \eqref{4.2} follows by a straightforward computation.
\end{proof}

\section{About the proofs of Theorems \ref{th:2.1} and \ref{th:3.1}}
\label{Section6}

We discuss in somewhat more detail the new ideas leading to the improved
bound for the number of subspaces in Theorem \ref{th:2.1} as compared
with Theorem B. For simplicity, we consider only the special case
$K=\Qq$, $S=\{\infty\}$, $O_S=\Zz$. Notice that 
$H({\bf x})=\|{\bf x}\|=\max (|x_1|\kdots |x_n|)$ for 
${\bf x}\in\Zz^n\setminus\{ {\bf 0}\}$.
Thus, we consider systems of inequalities
\begin{equation}\label{6.1}
|L_i({\bf x})|\leq C\cdot \|{\bf x}\|^{c_i}\ (i=1\kdots n)\ \ \mbox{in }
{\bf x}\in\Zz^n ,
\end{equation}
where $L_1\kdots L_n$ are linearly independent linear forms in
$X_1\kdots X_n$ with coefficients in $\Cc$ that are algebraic over $\Qq$,
$0<C\leq |\det (L_1\kdots L_n)|^{1/n}$, and $c_1+\cdots +c_n\leq -\delta$
with $0<\delta \leq 1$.

With a solution ${\bf x}\in\Zz^n$ we associate a convex body $\Pi ({\bf x})$,
consisting of those ${\bf y}\in\Rr^n$ such that
\[
|L_i({\bf y})|\leq C\|{\bf x}\|^{c_i}\ \mbox{for } i=1\kdots n.
\] 
Denote by $\lambda_i({\bf x})$ ($i=1\kdots n$) the successive minima of
this body. Then $\lambda_1({\bf x})\leq 1$, and by Minkowski's theorem,
$\prod_{i=1}^n \lambda_i({\bf x})\gg$ 
${\rm vol} \big(\Pi ({\bf x})\big)^{-1}\gg \|{\bf x}\|^{\delta}$, 
where here and below,
the constants implied by $\ll$, $\gg$ depend on $n$, $L_1\kdots L_n$ and
$\delta$.

There is an index $k\in\{ 1\kdots n-1\}$ such that
$\lambda_k({\bf x})/\lambda_{k+1}({\bf x})\ll \|{\bf x}\|^{-\delta /n}$.
To apply the approximation techniques going into the Subspace Theorem,
one needs
that the one but last minimum $\lambda_{n-1}({\bf x})$ is $\ll 1$.
In general, this need not be the case. 
Schmidt's ingenious idea was, to construct from $\Pi ({\bf x})$ 
a new convex
body $\widehat{\Pi}({\bf x})$ in $\wedge^{n-k}\Rr^n\cong \Rr^N$ 
with $N:={n\choose k}$
of which the one but last minimum is indeed $\ll 1$. The body
$\widehat{\Pi}({\bf x})$ may be described 
as the set of $\widehat{{\bf y}}\in\Rr^N$ such that
\begin{equation}\label{6.2}
|M_i(\widehat{{\bf y}})|\ll \|{\bf x}\|^{e_i({\bf x})}\ 
\mbox{for $i=1\kdots N$,}
\end{equation}
where $M_1\kdots M_N$ are linearly independent linear forms in $N$
variables with real algebraic coefficients, and 
$e_1({\bf x})\kdots e_N({\bf x})$
are exponents, which unfortunately may depend on ${\bf x}$, such that 
$\sum_{i=1}^N e_i({\bf x})< -\delta /2n^2$, say, see
\cite{Schm80} or \cite{Ev93} for more details on Schmidt's construction.
As mentioned before, the one but last minimum of $\widehat{\Pi}({\bf x})$ is
$\ll 1$. Then by Minkowski's Theorem,
the last minimum is $\gg \|{\bf x}\|^{\delta /2n^2}$.
This implies that $\widehat{\Pi}({\bf x})\cap\Zz^N$ spans
a linear subspace $T({\bf x})$ of $\Qq^N$ of dimension $N-1$.

In their proof of Theorem B, Schlickewei and the author had to partition
the set of solutions of \eqref{6.1} into classes in such a way,
that for any two solutions ${\bf x}$, ${\bf x}'$ in the same class, 
we have $e_i({\bf x})\approx e_i({\bf x}')$ for $i=1\kdots N$.
Then they proceeded further with solutions from the same class.

The continuation of the proof of Schlickewei and the 
author is then as follows.
Suppose there are solutions ${\bf x}_1\kdots {\bf x}_M$ in the same class 
such that $\|{\bf x}_1\|$ is large and $\log\|{\bf x}_{i+1}\|/\log\|{\bf x}_i\|$
are large for $i=1\kdots M-1$, where $M$ and ``large" depend on $\delta$, $n$
and $L_1\kdots L_n$.
Then one constructs an auxiliary 
multihomogeneous
polynomial $P({\bf Y}_1\kdots {\bf Y}_M)$ in $M$ blocks of $N$ variables
with integer coefficients, 
which is of degree $d_i$ in block ${\bf Y}_i$ for $i=1\kdots M$,
where $\|{\bf x}_1\|^{d_1}\approx\cdots\approx \|{\bf x}_M\|^{d_M}$.
The polynomial $P$ is
such that
$|P_I(\widehat{{\bf y}}_1\kdots \widehat{{\bf y}}_M)|<1$
for all $\widehat{{\bf y}}_h\in\widehat{\Pi}({\bf x}_h)\cap\Zz^N$, $h=1\kdots M$,
and all partial derivatives $P_I$ of $P$ of not too large order.
Then for these $I$, $\widehat{{\bf y}}_1\kdots \widehat{{\bf y}}_M$ we have that
$P_I(\widehat{{\bf y}}_1\kdots \widehat{{\bf y}}_M)=0$. By extrapolation
it then follows that all $P_I$ vanish identically on
$T({\bf x}_1)\times\cdots\times T({\bf x}_M)$. 
On the other hand, using an extension of Roth's Lemma, proved also 
by Schmidt, one shows that such a polynomial cannot exist.

This contradiction shows that solutions ${\bf x}_1\kdots{\bf x}_M$ as above
cannot exist. This leads to an upper bound 
depending on $n,\delta ,D$
for the number of subspaces
containing the solutions of \eqref{6.1} belonging to a given class.
We have to multiply this 
with the number
of classes
to get our final bound for the number of subspaces containing the
solutions from all classes together.
As it turns out,
the number of classes is at most
$\gamma_1^{n^2}\delta^{-\gamma_2n}$
with absolute constants $\gamma_1,\gamma_2$ and in terms of $n,\delta$,
this dominates the resulting bound
for the number of subspaces.

In their proof of Theorem \ref{th:2.1}, Ferretti and the author used,
instead of Schmidt's multi-homogeneous polynomial, the one
constructed by Faltings and W\"{u}stholz \cite{FW94}.
The latter polynomial has the great advantage, that the argument sketched above
works also for solutions ${\bf x}_1\kdots {\bf x}_M$ not necessarily
belonging to the same class. Thus, a subdivision of the solutions
of \eqref{6.1} into classes is not necessary, and we can save a factor
$\gamma_1^{n^2}\delta^{-\gamma_2n}$ in the final upper bound for the number of
subspaces.

The proof of the interval result Theorem \ref{th:3.1} follows the same lines.
First one proves Theorem \ref{th:3.1} in the special case that
the exceptional subspace
$U_0=({\bf 0})$. Assuming that Theorem \ref{th:3.1} is false,
one arrives at a contradiction using Schmidt's construction of
$\widehat{\Pi}({\bf x})$, Faltings' and W\"{u}stholz' construction
of an auxiliary polynomial, and Schmidt's extension of Roth's Lemma.
Then one proves the result for arbitrary $U_0$ 
by considering a system derived from
\eqref{6.1} with solutions taken from the quotient $\Qq^n/U_0$.          

We now discuss the constructions of an auxiliary polynomial 
by Schmidt and by Faltings and W\"{u}stholz, respectively.

We have to construct a non-zero multihomogeneous polynomial
\[
P({\bf Y}_1\kdots {\bf Y}_M)\in\Zz [{\bf Y}_1\kdots{\bf Y}_M]
\]
in $M$ blocks ${\bf Y}_1\kdots{\bf Y}_M$ of $N$
variables,
which is homogeneous of degree $d_h$ in the block ${\bf Y}_h$ for $h=1\kdots M$.
This polynomial can be expressed as
\[
\sum_{{\bf i}} c({\bf i})\prod_{h=1}^M\prod_{j=1}^N M_j({\bf Y}_h)^{i_{hj}}
\]
where the summation is over tuples ${\bf i}=(i_{hj})$ such that
$\sum_{j=1}^N i_{hj}=d_h$ for $h=1\kdots M$.

Schmidt's approach is to  
construct $P$ with coefficients with small absolute values,  such that
\[
c({\bf i})=0\ \mbox{if }\max_{1\leq j\leq N}
\left|\Big(\sum_{h=1}^M\frac{i_{hj}}{d_h}\Big)-\frac{M}{N}\right|
\geq\varepsilon
\]
for some sufficiently small $\varepsilon$. The conditions
$c({\bf i})=0$ may be viewed as linear equations in the unknown coefficients
of $P$. We may consider the indices $i_{hj}$ as random variables
with expectation $1/N$. Then the law of large numbers from probability
theory implies that for sufficiently large $M$,
the number of conditions $c({\bf i})=0$ is smaller than the total number
of coefficients of $P$. Now Siegel's Lemma gives a non-zero polynomial $P$
with coefficients with small absolute values.

The approach of Faltings and W\"{u}stholz is as follows.
Let $\alpha_{hj}\in\Rr$ with $|\alpha_{hj}|\leq 1$ for
$h=1\kdots m$, $j=1\kdots R$. Construct $P$ with coefficients
with small absolute values such that
\[
c({\bf i})=0\ \mbox{if }
\left|\sum_{h=1}^M\sum_{j=1}^N 
\alpha_{hj}\Big(\frac{i_{hj}}{d_h}-\frac{1}{N}\Big)\right|
\geq\varepsilon .
\]
Again, thanks to the law of large numbers, for sufficiently
large $M$ the number of conditions
$c({\bf i})=0$ is smaller than the number of coefficients of $P$,
and then $P$ is obtained via an application of Siegel's Lemma.

The choice of the weights $\alpha_{hj}$ is completely free.
In fact, if we are given solutions ${\bf x}_1\kdots {\bf x}_M$
of \eqref{6.1} from different classes, we may choose the $\alpha_{hj}$
in a suitable manner depending on the exponents $e_i({\bf x}_h)$
($i=1\kdots N$, $h=1\kdots M$) from \eqref{6.2},
and then show that $|P_I(\widehat{{\bf y}}_1\kdots \widehat{{\bf y}}_M)|<1$
for all ${\bf y}_h\in\widehat{\Pi}({\bf x}_h)\cap\Zz^N$, $h=1\kdots M$,
and all partial derivatives $P_I$ of $P$ of not too large order.
Then the proofs of Theorems \ref{th:2.1} and \ref{th:3.1}
are completed as sketched above.

\end{document}